\documentclass[11pt]{article}
\usepackage[T2A]{fontenc}
\usepackage[cp1251]{inputenc}
\usepackage{amsmath}
\usepackage{amssymb,amsthm}
\newcounter{propcounter}
\newtheorem{theorem}[propcounter]{Theorem}

\newtheorem{proposition}[propcounter]{Proposition}

\newcommand{\trace}{\mathop{\rm Tr}}
\newcommand{\ad}{\mathop{\rm ad}}
\newcommand{\riemann}{\mathop{\rm R}}

\makeatletter

\renewcommand{\section}{\@startsection{section}{1}{\parindent}
{3.5ex plus 1ex minus .2ex}{2.3ex plus .2ex}{\normalsize
\bfseries}}

\makeatother

\begin{document}

\begin{center}
\Large Submanifolds with the Harmonic Gauss Map in Lie Groups
\end{center}

\vskip 1.5mm

\begin{center}
\large E. V. Petrov\footnote{The author was partially supported by N. I. Akhiezer Foundation and the Foundation of
Fundamental Research of Ukraine (Project No. GP/F13/0019 for young scientists).}
\end{center}

\vskip 1.5mm

\begin{center}
\normalsize \itshape V. N. Karazin Kharkiv National University\\
4 Svobody sq., Kharkiv, 61077, Ukraine\\
E-mail: petrov@univer.kharkov.ua\\
\end{center}

\vskip 1.5mm

\begin{center}
\begin{minipage}[t]{130mm}
\small In this paper we find a criterion for the Gauss map of an immersed smooth submanifold in some Lie group with
left invariant metric to be harmonic. Using the obtained expression we prove some necessary and sufficient conditions
for the harmonicity of this map in the case of totally geodesic submanifolds in Lie groups admitting biinvariant
metrics. We show that, depending on the structure of the tangent space of a submanifold, the Gauss map can be harmonic
in all biinvariant metrics or non-harmonic in some metric. For $2$-step nilpotent groups we prove that the Gauss map of
a geodesic is harmonic if and only if it is constant.
\end{minipage}
\end{center}

\section{Introduction} \label{ch1}

It is proved in \cite{RV} that the Gauss map of a submanifold in the Euclidean space is harmonic if and only if the
mean curvature field of this submanifold is parallel. There is a natural generalization of the Gauss map to
submanifolds in Lie groups: for each point of a submanifold the tangent space at this point is translated to the
identity element of the group (for the precise statement see Section~\ref{ch2}). Let the Lie group be endowed with some
left invariant metric. As it is proved in \cite{ES}, when this metric is biinvariant and the submanifold is
hypersurface, the Gauss map is harmonic if and only if the mean curvature is constant. Our aim is to consider more
general case of a submanifold in some Lie group with arbitrary codimension.

The paper is organized as follows. In Section~\ref{ch2} we obtain the harmonicity criterion for the Gauss map of a
submanifold in some Lie group with a left invariant metric. This criterion is given in the terms of the second
fundamental form of the immersion and the left invariant Riemannian connection on the Lie group (Theorem~\ref{th1}).

In Section~\ref{ch3} we consider submanifolds in Lie groups with biinvariant metric. Let us introduce some notation.
Let $N$ be a Lie group with biinvariant metric, $\mathcal{N}$ the Lie algebra of $N$, $M$ a smooth immersed totally
geodesic submanifold in $N$. Taking if necessary the left translation of $M$ assume that $e\in M$ (see
Section~\ref{ch3} for details). The tangent space $T_e M$ is a Lie triple system in $\mathcal{N}$. Denote by
$\overline{\mathcal{N}}$ the Lie subalgebra $T_e M+[T_e M,T_e M]$ of $\mathcal{N}$. By $\mathcal{W}$ denote the
orthogonal projection of $T_e M$ to the semisimple Lie subalgebra
$\overline{\mathcal{N}}'=[\overline{\mathcal{N}},\overline{\mathcal{N}}]$.

The subspace $\overline{\mathcal{W}}=\mathcal{W}\cap[\mathcal{W},\mathcal{W}]$ is an ideal (here by ideals we mean
ideals in $\overline{\mathcal{N}}$). Denote by $\mathcal{V}$ the orthogonal complement in $\overline{\mathcal{N}}'$ to
$\overline{\mathcal{W}}$. Let $\mathcal{V}=\bigoplus\limits_{1\leqslant l \leqslant m}\mathcal{S}_l$ be some direct
orthogonal decomposition of $\mathcal{V}$ into simple ideals. Using Theorem~\ref{th1} we prove

\setcounter{propcounter}{3} \addtocounter{propcounter}{-1}

\begin{theorem}
Let $M$ be a smooth immersed totally geodesic submanifold in a Lie group $N$ with biinvariant metric. Then
\begin{enumerate}

\item if the restriction of the metric to $\mathcal{V}$ is a negative multiple of the Killing form (in particular, if
$\mathcal{V}$ is simple), then the Gauss map of $M$ in this metric is harmonic;

\item if $\mathcal{W}\cap\mathcal{V}=\bigoplus\limits_{1\leqslant l \leqslant m}\mathcal{W}_l$, where
$\mathcal{W}_l\subset\mathcal{S}_l$ is a proper Lie triple system in $\mathcal{S}_l$, i.e., $\mathcal{W}_l \neq 0$ and
$\mathcal{W}_l \neq \mathcal{S}_l$ for each $1\leqslant l \leqslant m$ (in particular, if $\mathcal{V}=0$), then the
Gauss map of $M$ is harmonic in any biinvariant metric on $N$;

\item if the condition of \ref{pr2-it2} is not satisfied, then there is a biinvariant metric on $N$ such that the Gauss
map of $M$ is not harmonic.

\end{enumerate}
\end{theorem}

\setcounter{propcounter}{0}

In the paper \cite{Pe} we considered hypersurfaces in $2$-step nilpotent Lie groups and found conditions for the Gauss
maps of such hypersurfaces to be harmonic. In particular, we showed that, unlike the case of groups with biinvariant
metric, this harmonicity is not equivalent to the constancy of the mean curvature. As it was shown in \cite{E2},
totally geodesic submanifolds in such groups either have the Gauss map of maximal rank or are open subsets of subgroups
(and consequently have the constant Gauss map). In the latter case the structure of subalgebras corresponding to such
subgroups can be explicitly described (this description implies, in particular, that there are not totally geodesic
hypersurfaces in $2$-step nilpotent Lie groups, see \cite{E2}). Using our criterion, we prove in Section~\ref{ch4} that
the Gauss map of a geodesic in a $2$-step nilpotent Lie group is harmonic if and only if is constant
(Proposition~\ref{pr3}).

The author would thank prof. A. A. Borisenko and prof. L. A. Masal'tsev for their attention to this work and many
useful remarks. Also the author is grateful to the reviewer for essential improvement of both the results and the
presentation of the paper.

\section{The Harmonicity Criterion} \label{ch2}

Suppose $M$ is a smooth manifold, $\dim M = n$, $M \rightarrow N$ is an immersion of $M$ in some Lie group $N$ with
left invariant metric, $\dim N=n+q$. For some point $p$ of $M$ let $Y_1,\dots,Y_n$ and $Y_{n+1},\dots,Y_{n+q}$ be
orthonormal frames of tangent space $T_p M \subset T_p N$ and of normal space $N_p M \subset T_p N$, respectively. Also
by $Y_a$, $1 \leqslant a \leqslant n+q$, denote the corresponding left invariant fields on $N$.

Denote the left invariant metric on $N$ (and also the corresponding inner product on its Lie algebra) by $\langle
\cdot,\cdot \rangle$, the Riemannian connection of this metric by $\nabla$, its curvature tensor by
$\riemann(\cdot,\cdot)\cdot$, and the normal connection of the immersion $M \rightarrow N$ by $\nabla^\bot$.

Let $E_1, \ldots E_{n+q}$ be vector fields defined on some neighborhood $U$ of $p$ such that $E_i(p) = Y_i$, $E_1,
\dots , E_n$ and $E_{n+1}, \dots , E_{n+q}$ are orthonormal frames of the tangent and the normal bundles of $M$ on $U$,
respectively, and $\left(\nabla _{E_i}E_j\right)^T(p)=0$, for all $1 \leqslant i,j \leqslant n$. Then the {\itshape
mean curvature field} $H$ of the immersion is defined on $U$ by
\begin{equation}\label{in00}
\begin{array}{c}
H=\frac{1}{n}\sum\limits_{1 \leqslant i \leqslant n}\left(\nabla_{E_i}E_i \right)^\bot.\\
\end{array}
\end{equation}
Here $(\cdot)^T$ and $(\cdot)^\bot$ are the projections to the tangent bundle $TM$ and the normal bundle $NM$,
respectively.

For $1 \leqslant i,j \leqslant n$, $n+1 \leqslant \alpha \leqslant n+q$ by $b_{ij}^\alpha=\langle \nabla_{E_i}E_j,
E_\alpha \rangle$ denote the coefficients of the second fundamental form of the immersion on $U$ with respect to the
frame $E_1, \dots, E_{n+q}$. Suppose that on $U$ for $1\leqslant a \leqslant n+q$
\begin{equation}\label{in0}
\begin{array}{c}
E_a = \sum \limits_{1\leqslant b \leqslant n+q} A^b_a Y_b.\\
\end{array}
\end{equation}
Here $\{A^b_a\}_{1\leqslant a,b \leqslant n+q}$ are functions on $U$. Obviously, $A^b_a(p)=\delta_{ab}$, where
$\delta_{ab}$ is the Kronecker symbol.

Let $\Delta$ be the Laplacian $\Delta_M$ of the induced metric on $M$. The definition of the Laplacian and the
conditions $\left(\nabla _{E_i}E_j\right)^T(p)=0$ imply that for functions $f$ and $g$ defined on $U$
\begin{equation}\label{in1}
\begin{array}{c}
\Delta f(p)=\sum\limits_{1\leqslant i \leqslant n}E_iE_i(f),\\
\end{array}
\end{equation}
\begin{equation}\label{in2}
\begin{array}{c}
\Delta (fg)(p)=g(p)\Delta f(p)+2\sum\limits_{1\leqslant i \leqslant n}E_i(f)E_i(g)+f(p)\Delta g(p).\\
\end{array}
\end{equation}
Let $\Phi$ be the {\itshape Gauss map} of $M$:
\begin{equation}\label{in3}
\begin{array}{c}
\Phi\colon M \rightarrow G(n,q); \Phi(p)=dL_{p^{-1}}(T_p M).\\
\end{array}
\end{equation}
Here $G(n,q)$ is the Grassmannian of $n$-dimensional subspaces in $n+q$-dimensional vector space, a point $p$ is
identified with its image under the immersion, $L_g$ is the left translation by $g \in M$, $dF$ is the differential of
a map $F$.

Recall that if $(M_1,g_1)$ and $(M_2,g_2)$ are smooth Riemannian manifolds, then for any $\phi \in C^{\infty}(M_1,M_2)$
the {\itshape energy} of $\phi$ is
$$
\begin{array}{c}
E(\phi)=\frac{1}{2}\int\limits_{M_1} \sum\limits_{1 \leqslant i \leqslant m}g_2(d\phi( E_i), d\phi ( E_i)) dV_M,\\
\end{array}
$$
where $m= \dim M_1$, $E_1, \dots , E_m$ is the orthonormal frame on $M_1$, $dV_M$ is the volume form of $g_1$. The
critical points of the functional $\phi \mapsto E(\phi)$ are called {\itshape harmonic maps} from $M_1$ to $M_2$. We
say that a map is harmonic at some point if the corresponding Euler-Lagrange equations are satisfied at this point
(i.e., the so-called {\itshape tension field} vanishes, see, for example, \cite{Ur}).

\begin{theorem}\label{th1}
The map $\Phi$ is harmonic at $p$ if and only if
\begin{equation}\label{eq1}
\begin{array}{c}
\sum\limits_{1 \leqslant i \leqslant n}\langle \riemann(Y_j,Y_i)Y_i,Y_\alpha\rangle
-\sum\limits_{1\leqslant i \leqslant n}\langle\nabla_{\left(\nabla_{Y_i}Y_i\right)}Y_j, Y_\alpha \rangle
+\langle [nH,Y_j],Y_\alpha\rangle\\
+2\sum\limits_{1\leqslant i,k \leqslant n}b_{ik}^\alpha\langle\nabla_{Y_i}Y_k, Y_j \rangle+2\sum\limits_{1 \leqslant i
\leqslant n,n+1\leqslant \gamma\leqslant n+q}b^\gamma_{ij}\langle\nabla_{Y_i}Y_\gamma,Y_\alpha\rangle\\
-\sum\limits_{1\leqslant i\leqslant n}\langle\left(\nabla_{Y_i}Y_j\right)^T,\left(\nabla_{Y_i}Y_\alpha\right)^T
\rangle+\sum\limits_{1\leqslant i\leqslant n}\langle\left(\nabla_{Y_i}Y_j\right)^\bot,
\left(\nabla_{Y_i}Y_\alpha\right)^\bot \rangle=0\\
\end{array}
\end{equation}
for all $1\leqslant j\leqslant n$, $n+1\leqslant \alpha\leqslant n+q$.
\end{theorem}
\begin{proof}
The Grassmannian has the structure of the symmetric space $G(n,q)=O(n+q)/\left(O(n) \times O(q)\right)$. There is an
embedding of this space in the space of symmetric matrices of order $n+q$ considered with the obvious Euclidean metric
(\cite{Ko}). This embedding is induced by the map $A \mapsto AEA^{t}$, where $A \in O(n+q)$, $A^{t}$ is $A$ transposed,
and
$$
\begin{array}{c}
E=\left(
\begin{array}{cc}
-\frac{q}{n+q}I_n & 0 \\
0 & \frac{n}{n+q}I_q \\
\end{array}
\right).\\
\end{array}
$$
Here $I_n$ and $I_q$ are the identity matrices of order $n$ and $q$, respectively. The image of $\Phi$ on $U$
corresponds to the matrix $A=(A^b_a)_{1\leqslant a,b \leqslant n+q}$, where $A^b_a$ are the functions from \eqref{in0}.
The composition of $\Phi$ and the embedding gives the map defined on $U$ by
\begin{equation}\label{eq1-pr1-1}
\begin{array}{c}
\left(
\begin{array}{cc}
-\frac{q}{n+q}I_n+\left(\sum\limits_{n+1\leqslant \gamma \leqslant
n+q}A^j_\gamma A^k_\gamma \right) &
\left(\sum\limits_{n+1\leqslant \gamma \leqslant n+q}A^j_\gamma
A^\beta_\gamma \right)\\
\left(\sum\limits_{n+1\leqslant \gamma \leqslant
n+q}A^\alpha_\gamma A^k_\gamma \right) & \frac{n}{n+q}I_q-
\left(\sum\limits_{1\leqslant l \leqslant n}A^\alpha_l
A^\beta_l \right)\\
\end{array}
\right),\\
\end{array}
\end{equation}
where $1\leqslant j,k \leqslant n$, $n+1 \leqslant \alpha,\beta \leqslant n+q$. Differentiate $E_a$ with respect to
$E_i$ on $U$ for $1\leqslant a \leqslant n+q$, $1\leqslant i\leqslant n$:
\begin{equation}\label{eq1-pr3}
\begin{array}{c}
\nabla_{E_i}E_a=\sum\limits_{1 \leqslant b \leqslant n+q}E_i(A^b_a)Y_b+ \sum\limits_{1 \leqslant b \leqslant
n+q}A^b_a\nabla_{E_i}Y_b.\\
\end{array}
\end{equation}
In particular, at $p$
\begin{equation}\label{eq1-pr4}
\begin{array}{c}
\nabla_{Y_i}E_a=\sum\limits_{1 \leqslant b \leqslant n+q}E_i(A^b_a)Y_b+ \nabla_{Y_i}Y_a.\\
\end{array}
\end{equation}
Note that $E_i(A^b_a)=-E_i(A^a_b)$ (this can be derived from \eqref{eq1-pr4} or simply from the fact that
$\mathfrak{so}(n+q)$ is the algebra of skew-symmetric matrices).

According to Theorem~(2.22) in Chapter~4 of \cite{Ur}, the criterion of the harmonicity of $\Phi$ is the set of
equations
\begin{equation}\label{eq1-pr1-2}
\begin{array}{c}
\Delta \Phi^a_b-\left( \sum\limits_{1\leqslant i\leqslant n}B(d\Phi(E_i),d\Phi(E_i)) \right)^a_b=0.\\
\end{array}
\end{equation}
Here $1\leqslant a \leqslant b \leqslant n+q$, $\Phi^a_b$ are the coordinate functions of the embedding, and $B$ is the
second fundamental form of the embedding\footnote{Actually, the sign of $B$ in \cite{Ur} is different because the
Laplacian in this book is defined with the opposite sign.}. The fields $\{\frac{\partial}{\partial
A^j_\alpha}\}_{1\leqslant j \leqslant n,n+1\leqslant \alpha \leqslant n+q}$ form the frame of $TG(n,q)$ on the image of
$U$ (note that $\frac{\partial}{\partial A^\alpha_j}=-\frac{\partial}{\partial A^j_\alpha}$). Denote by $C^a_b$ for $1
\leqslant a \leqslant b \leqslant n+q$ the matrix with the entry $1$ at the intersection of $a-$th row and $b-$th
column and with other entries equal to $0$. The differential of the embedding at $p$ maps the field
$\frac{\partial}{\partial A^j_\alpha}$ to the vector $C^j_\alpha$. It follows that we can take as a frame of the normal
space of the Grassmannian at the image of this point the vectors $C^i_j$, $1\leqslant i\leqslant j \leqslant n$ and
$C^\alpha_\beta$, $n+1\leqslant \alpha\leqslant \beta \leqslant n+q$. The expressions \eqref{eq1-pr1-1} imply on $U$
for $1\leqslant l\leqslant m \leqslant n$, $n+1\leqslant \gamma\leqslant \kappa \leqslant n+q$
$$
\begin{array}{c}
\left(\frac{\partial}{\partial
A^j_\alpha}\right)^l_m=\delta_{lj}A^m_\alpha+\delta_{mj}A^l_\alpha,\left(\frac{\partial}{\partial
A^j_\alpha}\right)^\gamma_\kappa=\delta_{\gamma\alpha}A^\kappa_j+\delta_{\kappa\alpha}A^\gamma_j.\\
\end{array}
$$
Differentiate these equations:
$$
\begin{array}{c}
B\left(\frac{\partial}{\partial A^j_\alpha},\frac{\partial}{\partial A^k_\beta}\right)=\sum\limits_{1\leqslant
l\leqslant m \leqslant n}\frac{\partial}{\partial A^k_\beta}\left(\frac{\partial}{\partial A^j_\alpha}\right)^l_m
C^l_m\\
+\sum\limits_{n+1\leqslant \gamma\leqslant \kappa \leqslant n+q}\frac{\partial}{\partial
A^k_\beta}\left(\frac{\partial}{\partial A^j_\alpha}\right)^\gamma_\kappa C^\gamma_\kappa=\delta_{\alpha
\beta}(1+\delta_{jk})C^j_k-\delta_{jk}(1+\delta_{\alpha \beta})C^\alpha_\beta.\\
\end{array}
$$
for $1\leqslant j \leqslant k \leqslant n$, $n+1\leqslant \alpha \leqslant \beta \leqslant n+q$. Also note that
$$
\begin{array}{c}
d\Phi(E_i)=\sum\limits_{1\leqslant j \leqslant n,n+1\leqslant \alpha \leqslant
n+q}E_i(A^j_\alpha)\frac{\partial}{\partial A^j_\alpha}.\\
\end{array}
$$
This implies that at $p$ for $1\leqslant j \leqslant k \leqslant n$ the expressions in \eqref{eq1-pr1-2} take the form
$$
\begin{array}{c}
\Delta \left(-\frac{q}{n+q}\delta_{jk}+\sum\limits_{n+1\leqslant \gamma \leqslant n+q}A^j_\gamma A^k_\gamma
\right)-\left( \sum\limits_{1\leqslant i\leqslant n}B(d\Phi(E_i),d\Phi(E_i)) \right)^j_k\\
=2\sum\limits_{1\leqslant i \leqslant n,n+1\leqslant \gamma \leqslant n+q}E_i(A^j_\gamma)
E_i(A^k_\gamma)-2\sum\limits_{1\leqslant i \leqslant n,n+1\leqslant \gamma \leqslant n+q}E_i(A^j_\gamma)
E_i(A^k_\gamma)=0.\\
\end{array}
$$
Here the equation \eqref{in2} was used. Similarly, for $n+1\leqslant \alpha \leqslant \beta \leqslant n+q$ obtain
$$
\begin{array}{c}
\Delta \left( \frac{n}{n+q}\delta_{\alpha\beta}- \sum\limits_{1\leqslant l \leqslant n}A^\alpha_l
A^\beta_l\right)-\left( \sum\limits_{1\leqslant i\leqslant n}B(d\Phi(E_i),d\Phi(E_i)) \right)^\alpha_\beta\\
=-2\sum\limits_{1\leqslant i,l \leqslant n}E_i(A^\alpha_l)E_i(A^\beta_l) +2\sum\limits_{1\leqslant i,l \leqslant
n}E_i(A^\alpha_l)E_i(A^\beta_l)=0.\\
\end{array}
$$
It follows that the conditions \eqref{eq1-pr1-2} at $p$ become
$$
\begin{array}{c}
\Delta\left(\sum\limits_{n+1\leqslant \gamma \leqslant n+q}A^j_\gamma A^\beta_\gamma \right)=0\\
\end{array}
$$
for $1\leqslant j \leqslant n, \, n+1\leqslant \beta \leqslant n+q$. The differentiation gives
\begin{equation}\label{eq1-pr2}
\begin{array}{c}
\Delta A^j_\alpha+2\sum\limits_{1 \leqslant i \leqslant n, n+1
\leqslant \gamma \leqslant
n+q}E_i(A^j_\gamma)E_i(A^\alpha_\gamma)=0;\\
\quad 1\leqslant j \leqslant n, \, n+1\leqslant \alpha \leqslant n+q.\\
\end{array}
\end{equation}
Differentiate \eqref{eq1-pr3} with respect to $E_i$ on $U$ for $n+1 \leqslant a=\alpha \leqslant n+q$:
\begin{equation}\label{eq1-pr6}
\begin{array}{c}
\nabla_{E_i}\nabla_{E_i}E_\alpha=\sum\limits_{1\leqslant j
\leqslant n}E_iE_i(A^j_\alpha)Y_j+ \sum\limits_{n+1\leqslant \beta
\leqslant n+q}E_iE_i(A^\beta_\alpha)Y_\beta
\\+2\sum\limits_{1\leqslant j \leqslant n}E_i(A^j_\alpha)\nabla_{E_i}Y_j
+2\sum\limits_{n+1\leqslant \beta \leqslant n+q}E_i(A^\beta_\alpha)\nabla_{E_i}Y_\beta\\
+\sum\limits_{1\leqslant j \leqslant n}A^j_\alpha\nabla_{E_i}\nabla_{E_i}Y_j+ \sum\limits_{n+1\leqslant \beta \leqslant
n+q}A^\beta_\alpha\nabla_{E_i}\nabla_{E_i}Y_\beta.
\end{array}
\end{equation}
Take the inner product of \eqref{eq1-pr6} with $Y_j$ at $p$:
$$
\begin{array}{c}
E_iE_i(A^j_\alpha)=\langle \nabla_{E_i}\nabla_{E_i}E_\alpha, Y_j \rangle-2\sum\limits_{1\leqslant k \leqslant
n}E_i(A^k_\alpha)\langle\nabla_{E_i}Y_k, Y_j \rangle\\
-2\sum\limits_{n+1\leqslant \gamma \leqslant n+q}E_i(A^\gamma_\alpha)\langle\nabla_{E_i}Y_\gamma, Y_j \rangle- \langle
\nabla_{E_i}\nabla_{E_i}Y_\alpha, Y_j \rangle.\\
\end{array}
$$
Therefore \eqref{eq1-pr2} takes the form
\begin{equation}\label{eq1-pr7}
\begin{array}{c}
\sum\limits_{1\leqslant i \leqslant n}\langle
\nabla_{E_i}\nabla_{E_i}E_\alpha, Y_j
\rangle-2\sum\limits_{1\leqslant i,k \leqslant
n}E_i(A^k_\alpha)\langle\nabla_{E_i}Y_k, Y_j \rangle\\
-2\sum\limits_{1\leqslant i \leqslant n,n+1\leqslant \gamma
\leqslant n+q }E_i(A^\gamma_\alpha)\langle\nabla_{E_i}Y_\gamma,
Y_j \rangle-\sum\limits_{1\leqslant i \leqslant
n} \langle \nabla_{E_i}\nabla_{E_i}Y_\alpha, Y_j \rangle\\
+2\sum\limits_{1 \leqslant i \leqslant n, n+1\leqslant \gamma
\leqslant n+q}E_i(A^j_\gamma)E_i(A^\alpha_\gamma)=0.
\end{array}
\end{equation}
Here \eqref{in1} was used. The definition \eqref{in00} of the mean curvature field implies for $1 \leqslant j \leqslant
n$, $n+1 \leqslant \alpha \leqslant n+q$ at $p$
$$
\begin{array}{c}
\langle \nabla_{Y_j}(nH),Y_\alpha\rangle=\sum\limits_{1 \leqslant i \leqslant n}\langle \nabla_{E_j}\left(
\left(\nabla_{E_i}E_i \right)^\bot \right),E_\alpha\rangle=\sum\limits_{1 \leqslant i \leqslant n}\langle
\nabla_{E_j}\nabla_{E_i}E_i,E_\alpha\rangle\\
-\sum\limits_{1 \leqslant i \leqslant n}\langle \nabla_{E_j}\left( \left(\nabla_{E_i}E_i \right)^T
\right),E_\alpha\rangle= \sum\limits_{1 \leqslant i \leqslant n}\langle
\riemann(E_j,E_i)E_i+\nabla_{E_i}\nabla_{E_j}E_i\\
+\nabla_{[E_j,E_i]}E_i,E_\alpha\rangle -\sum\limits_{1 \leqslant i \leqslant n}E_j\langle \left(\nabla_{E_i}E_i
\right)^T ,E_\alpha\rangle+\sum\limits_{1 \leqslant i \leqslant n}\langle \left(\nabla_{E_i}E_i \right)^T
,\nabla_{E_j}E_\alpha\rangle\\
=\sum\limits_{1 \leqslant i \leqslant n}\langle \riemann(Y_j,Y_i)Y_i,Y_\alpha\rangle +\sum\limits_{1 \leqslant i
\leqslant n}\langle \nabla_{E_i}\nabla_{E_j}E_i,E_\alpha\rangle=\sum\limits_{1 \leqslant i \leqslant n}\langle
\riemann(Y_j,Y_i)Y_i,Y_\alpha\rangle\\
+\sum\limits_{1 \leqslant i \leqslant n}\langle \nabla_{E_i}[E_j,E_i],E_\alpha\rangle +\sum\limits_{1 \leqslant i
\leqslant n}\langle \nabla_{E_i}\nabla_{E_i}E_j,E_\alpha\rangle =\sum\limits_{1 \leqslant i \leqslant n}\langle
\riemann(Y_j,Y_i)Y_i,Y_\alpha\rangle\\
+\sum\limits_{1 \leqslant i \leqslant n}E_i\langle [E_j,E_i],E_\alpha\rangle -\sum\limits_{1 \leqslant i \leqslant
n}\langle [E_j,E_i],\nabla_{E_i}E_\alpha\rangle+\sum\limits_{1 \leqslant i \leqslant n}\langle
\nabla_{E_i}\nabla_{E_i}E_j,E_\alpha\rangle\\
=\sum\limits_{1 \leqslant i \leqslant n}\langle \riemann(Y_j,Y_i)Y_i,Y_\alpha\rangle+\sum\limits_{1 \leqslant i
\leqslant n}\langle \nabla_{E_i}\nabla_{E_i}E_j,E_\alpha\rangle.\\
\end{array}
$$
In the third equality the definition of the curvature tensor was used. The fourth equality follows from the Frobenius
theorem, the condition $\left(\nabla_{E_i} E_j\right)^T(p)=0$, and its consequence
$$
\begin{array}{c}
[E_k,E_i](p)=\left([E_k,E_i]\right)^{T}(p)=\left(\nabla_{E_k}E_i-\nabla_{E_i}E_k\right)^{T}(p)=0.\\
\end{array}
$$
Differentiate two times the expression $\langle E_j,E_\alpha \rangle=0$ with respect to $E_i$:
$$
\begin{array}{c}
\langle \nabla_{E_i}\nabla_{E_i}E_j,E_\alpha\rangle+2\langle \nabla_{E_i}E_j,\nabla_{E_i}E_\alpha\rangle+\langle
E_j,\nabla_{E_i}\nabla_{E_i}E_\alpha\rangle=0.\\
\end{array}
$$
This equation and \eqref{eq1-pr4} imply
$$
\begin{array}{c}
\langle \nabla_{Y_j}(nH),Y_\alpha\rangle=\sum\limits_{1 \leqslant i \leqslant n}\langle
\riemann(Y_j,Y_i)Y_i,Y_\alpha\rangle-2\sum\limits_{1 \leqslant i \leqslant n}\langle\nabla_{E_i}E_j
,\nabla_{E_i}E_\alpha\rangle\\
-\sum\limits_{1 \leqslant i \leqslant n}\langle E_j,\nabla_{E_i}\nabla_{E_i}E_\alpha\rangle =\sum\limits_{1 \leqslant i
\leqslant n}\langle \riemann(Y_j,Y_i)Y_i,Y_\alpha\rangle\\
-2\sum\limits_{1 \leqslant i \leqslant n,n+1\leqslant \gamma\leqslant n+q}b^\gamma_{ij}\left(E_i(A^\gamma_\alpha)
+\langle\nabla_{Y_i}Y_\alpha,Y_\gamma\rangle\right)-\sum\limits_{1 \leqslant i \leqslant n}\langle
E_j,\nabla_{E_i}\nabla_{E_i}E_\alpha\rangle.\\
\end{array}
$$
From \eqref{eq1-pr4} and the condition $\left(\nabla_{E_i} E_j\right)^T(p)=0$ obtain
\begin{equation}\label{eq1-pr8}
\begin{array}{c}
b^\gamma_{ij}=E_i(A^\gamma_j)+\langle \nabla_{Y_i}Y_j,Y_\gamma \rangle,\\
\end{array}
\end{equation}
\begin{equation}\label{eq1-pr9}
\begin{array}{c}
0=E_i(A^k_j)+\langle \nabla_{Y_i}Y_j,Y_k \rangle.\\
\end{array}
\end{equation}
Hence at $p$
\begin{equation}\label{eq1-pr10}
\begin{array}{c}
\langle \nabla_{Y_j}(nH),Y_\alpha\rangle=\sum\limits_{1 \leqslant i \leqslant n}\langle
\riemann(Y_j,Y_i)Y_i,Y_\alpha\rangle\\
+2\sum\limits_{1 \leqslant i \leqslant n,n+1\leqslant
\gamma\leqslant
n+q}b^\gamma_{ij}\langle\nabla_{Y_i}Y_\gamma,Y_\alpha\rangle\\
-2\sum\limits_{1 \leqslant i \leqslant
n,n+1\leqslant \gamma\leqslant n+q}\langle \nabla_{Y_i}Y_j,Y_\gamma \rangle E_i(A^\gamma_\alpha)\\
-\sum\limits_{1 \leqslant i \leqslant n}\langle
E_j,\nabla_{E_i}\nabla_{E_i}E_\alpha\rangle -2\sum\limits_{1
\leqslant i \leqslant n,n+1\leqslant \gamma\leqslant
n+q}E_i(A^\gamma_j)E_i(A^\gamma_\alpha).\\
\end{array}
\end{equation}
The equation \eqref{eq1-pr8} implies
$$
\begin{array}{c}
\sum\limits_{1\leqslant i,k \leqslant n}E_i(A^k_\alpha)\langle\nabla_{E_i}Y_k, Y_j \rangle=-\sum\limits_{1\leqslant i,k
\leqslant n}E_i(A_k^\alpha)\langle\nabla_{E_i}Y_k, Y_j \rangle\\
=-\sum\limits_{1\leqslant i,k \leqslant n}b_{ik}^\alpha\langle\nabla_{Y_i}Y_k, Y_j \rangle+\sum\limits_{1\leqslant i,k
\leqslant n}\langle \nabla_{Y_i}Y_k,Y_\alpha\rangle\langle\nabla_{Y_i}Y_k, Y_j \rangle.\\
\end{array}
$$
Note that for each pair of left invariant fields $X$ and $Y$ the product $\langle X,Y\rangle$ is constant, hence,
\begin{equation}\label{eq1-pr10-0}
\begin{array}{c}
\langle \nabla_Z X,Y\rangle=Z\left(\langle X,Y\rangle\right)-\langle X,\nabla_Z Y\rangle=-\langle X,\nabla_Z Y\rangle\\
\end{array}
\end{equation}
for every vector $Z$. This and the fact that the frame is orthonormal imply
$$
\begin{array}{c}
\sum\limits_{1\leqslant i,k \leqslant n}\langle \nabla_{Y_i}Y_k,Y_\alpha\rangle\langle\nabla_{Y_i}Y_k, Y_j
\rangle=\sum\limits_{1\leqslant i,k \leqslant n}\langle \nabla_{Y_i}Y_j,Y_k\rangle\langle\nabla_{Y_i}Y_\alpha, Y_k
\rangle\\
=\sum\limits_{1\leqslant i\leqslant n}\langle\left(\nabla_{Y_i}Y_j\right)^T,\left(\nabla_{Y_i}Y_\alpha\right)^T
\rangle.\\
\end{array}
$$
Thus,
\begin{equation}\label{eq1-pr10-1}
\begin{array}{c}
\sum\limits_{1\leqslant i,k \leqslant n}E_i(A^k_\alpha)\langle\nabla_{E_i}Y_k, Y_j \rangle=-\sum\limits_{1\leqslant i,k
\leqslant n}b_{ik}^\alpha\langle\nabla_{Y_i}Y_k, Y_j \rangle\\
+\sum\limits_{1\leqslant i\leqslant n}
\langle\left(\nabla_{Y_i}Y_j\right)^T,\left(\nabla_{Y_i}Y_\alpha\right)^T \rangle.\\
\end{array}
\end{equation}
Substituting \eqref{eq1-pr10} in \eqref{eq1-pr7} and taking into account \eqref{eq1-pr10-1} derive the conditions
\begin{equation}\label{eq1-pr11}
\begin{array}{c}
\sum\limits_{1 \leqslant i \leqslant n}\langle \riemann(Y_j,Y_i)Y_i,Y_\alpha\rangle-\langle
\nabla_{Y_j}(nH),Y_\alpha\rangle+ 2\sum\limits_{1\leqslant i,k \leqslant n}b_{ik}^\alpha\langle\nabla_{Y_i}Y_k, Y_j
\rangle\\
+2\sum\limits_{1 \leqslant i \leqslant n,n+1\leqslant
\gamma\leqslant
n+q}b^\gamma_{ij}\langle\nabla_{Y_i}Y_\gamma,Y_\alpha\rangle-\sum\limits_{1\leqslant
i \leqslant n}
\langle \nabla_{E_i}\nabla_{E_i}Y_\alpha, Y_j \rangle\\
-2\sum\limits_{1\leqslant i\leqslant n}\langle\left(\nabla_{Y_i}Y_j\right)^T,\left(\nabla_{Y_i}Y_\alpha\right)^T
\rangle=0.\\
\end{array}
\end{equation}
At $p$ for $1 \leqslant i,j \leqslant n$, $n+1 \leqslant \alpha \leqslant n+q$ obtain
$$
\begin{array}{c}
\langle \nabla_{E_i}\nabla_{E_i}Y_\alpha, Y_j \rangle=\langle \nabla_{E_i}\left(\sum\limits_{1\leqslant a \leqslant
n+q}A_i^a\nabla_{Y_a}Y_\alpha\right), Y_j \rangle=\sum\limits_{1\leqslant k \leqslant
n}E_i(A_i^k)\langle\nabla_{Y_k}Y_\alpha, Y_j \rangle\\
+ \sum\limits_{n+1\leqslant \gamma \leqslant n+q}E_i(A_i^\gamma)\langle\nabla_{Y_\gamma}Y_\alpha, Y_j \rangle+\langle
\nabla_{Y_i}\nabla_{Y_i}Y_\alpha, Y_j \rangle.\\
\end{array}
$$
Substitute into this \eqref{eq1-pr8} and \eqref{eq1-pr9} and use the definition of the mean curvature $\langle
nH,E_\gamma\rangle=\sum\limits_{1 \leqslant i \leqslant n}b_{ii}^\gamma$. Then use \eqref{eq1-pr10-0}:
$$
\begin{array}{c}
\sum\limits_{1\leqslant i \leqslant n}\langle \nabla_{E_i}\nabla_{E_i}Y_\alpha, Y_j \rangle=-\sum\limits_{1\leqslant i
\leqslant n,1\leqslant a \leqslant n+q}\langle\nabla_{Y_i}Y_i,Y_a\rangle\langle\nabla_{Y_a}Y_\alpha, Y_j \rangle\\
+ \sum\limits_{n+1\leqslant \gamma \leqslant n+q}\langle nH, Y_\gamma\rangle\langle\nabla_{Y_\gamma}Y_\alpha, Y_j
\rangle+\sum\limits_{1\leqslant i \leqslant n}\langle \nabla_{Y_i}\nabla_{Y_i}Y_\alpha, Y_j \rangle\\
=\sum\limits_{1\leqslant i \leqslant n,1\leqslant a \leqslant
n+q}\langle\nabla_{Y_i}Y_i,Y_a\rangle\langle\nabla_{Y_a}Y_j, Y_\alpha \rangle\\
- \sum\limits_{n+1\leqslant \gamma \leqslant n+q}\langle nH, Y_\gamma\rangle\langle\nabla_{Y_\gamma}Y_j, Y_\alpha
\rangle-\sum\limits_{1\leqslant i \leqslant n}\langle \nabla_{Y_i}Y_j,\nabla_{Y_i}Y_\alpha\rangle.\\
\end{array}
$$
The frame is orthonormal, hence,
\begin{equation}\label{eq1-pr13}
\begin{array}{c}
\sum\limits_{1\leqslant i \leqslant n}\langle \nabla_{E_i}\nabla_{E_i}Y_\alpha, Y_j \rangle=\sum\limits_{1\leqslant i
\leqslant n}\langle\nabla_{\left(\nabla_{Y_i}Y_i\right)}Y_j, Y_\alpha \rangle\\
- \langle\nabla_{(nH)}Y_j, Y_\alpha
\rangle-\sum\limits_{1\leqslant i \leqslant n}\langle \nabla_{Y_i}Y_j,\nabla_{Y_i}Y_\alpha\rangle.\\
\end{array}
\end{equation}
Substitute \eqref{eq1-pr13} in \eqref{eq1-pr11} and obtain \eqref{eq1}.
\end{proof}

Note that the Gauss map of a Lie subgroup is constant, therefore harmonic.

If $N$ is the Euclidean space $E^{n+q}$, then the curvature tensor vanishes. For any vector field $X$ and for left
invariant (i.e., constant) $Y$ the derivatives $\nabla_X Y$ also vanish. This yields that the conditions \eqref{eq1}
take the form $\langle \nabla_{Y_j}(nH),Y_\alpha\rangle=0$ for $1\leqslant j\leqslant n$, $n+1\leqslant \alpha\leqslant
n+q$, i.e., $\nabla^\bot H=0$, and we obtain the above-mentioned classical result of \cite{RV}.

The definition of the second fundamental form and the fact that the frame is orthonormal allow us to rewrite
\eqref{eq1} in the form
\begin{equation}\label{eq1-1}
\begin{array}{c}
\sum\limits_{1 \leqslant i \leqslant n}\langle \riemann(Y_j,Y_i)Y_i,Y_\alpha\rangle -\sum\limits_{1\leqslant i
\leqslant n}\langle\nabla_{\left(\nabla_{Y_i}Y_i\right)}Y_j, Y_\alpha \rangle
+\langle [nH,Y_j],Y_\alpha\rangle\\
-2\sum\limits_{1\leqslant i \leqslant n}\langle \nabla_{\left(\nabla_{Y_i}Y_j\right)^T}E_i,
Y^\alpha\rangle-2\sum\limits_{1 \leqslant i
\leqslant n,}\langle \left(\nabla_{Y_i}E_j\right)^\bot, \left(\nabla_{Y_i}Y_\alpha\right)^\bot\rangle\\
-\sum\limits_{1\leqslant i\leqslant n}\langle\left(\nabla_{Y_i}Y_j\right)^T,\left(\nabla_{Y_i}Y_\alpha\right)^T
\rangle+\sum\limits_{1\leqslant i\leqslant n}\langle\left(\nabla_{Y_i}Y_j\right)^\bot,
\left(\nabla_{Y_i}Y_\alpha\right)^\bot \rangle=0.\\
\end{array}
\end{equation}
Note that these expressions do not depend on the particular choice of $E_1, \dots, E_n$.

The summands in \eqref{eq1} that do not include the coefficients of the second fundamental form and the mean curvature
field can be rewritten:
$$
\begin{array}{c}
\sum\limits_{1 \leqslant i \leqslant n}\langle \riemann(Y_j,Y_i)Y_i,Y_\alpha\rangle-\sum\limits_{1\leqslant i \leqslant
n}\langle\nabla_{\left(\nabla_{Y_i}Y_i\right)}Y_j, Y_\alpha \rangle\\
-\sum\limits_{1\leqslant i\leqslant n}\langle\left(\nabla_{Y_i}Y_j\right)^T,\left(\nabla_{Y_i}Y_\alpha\right)^T
\rangle+\sum\limits_{1\leqslant i\leqslant n}\langle\left(\nabla_{Y_i}Y_j\right)^\bot,
\left(\nabla_{Y_i}Y_\alpha\right)^\bot \rangle\\
=\sum\limits_{1 \leqslant i \leqslant n}\langle \nabla_{Y_j}\nabla_{Y_i}Y_i-\nabla_{\left(\nabla_{Y_i}Y_i\right)}Y_j-
\nabla_{Y_i}\nabla_{Y_j}Y_i-\nabla_{[Y_j,Y_i]}Y_i\\
\vphantom{\sum\limits_{1 \leqslant i \leqslant n}}
+\nabla_{Y_i}\left(\nabla_{Y_j}Y_i+[Y_i,Y_j]\right)^T-\nabla_{Y_i}\left(\nabla_{Y_j}Y_i+[Y_i,Y_j]\right)^\bot, Y_\alpha
\rangle\\
=\sum\limits_{1 \leqslant i \leqslant n}\langle[Y_j,\nabla_{Y_i}Y_i]- \nabla_{Y_i}\nabla_{Y_j}Y_i
-\nabla_{[Y_j,Y_i]}Y_i+\nabla_{Y_i}[Y_j,Y_i]\\
\vphantom{\sum\limits_{1 \leqslant i \leqslant n}}
+\nabla_{Y_i}\left(\nabla_{Y_j}Y_i+2[Y_i,Y_j]\right)^T-\nabla_{Y_i}\left(\nabla_{Y_j}Y_i\right)^\bot, Y_\alpha
\rangle\\
=\sum\limits_{1 \leqslant i \leqslant n}\langle[Y_j,\nabla_{Y_i}Y_i]+[Y_i,[Y_j,Y_i]]
+2\nabla_{Y_i}\left(\left([Y_i,Y_j]\right)^T-\left(\nabla_{Y_j}Y_i\right)^\bot\right), Y_\alpha \rangle.\\
\end{array}
$$
In particular, a totally geodesic submanifold $M$ has the harmonic Gauss map at $p$ if and only if
\begin{equation}\label{eq1-2}
\begin{array}{c}
\sum\limits_{1 \leqslant i \leqslant n}\left([Y_j,\nabla_{Y_i}Y_i]+[Y_i,[Y_j,Y_i]]
+2\nabla_{Y_i}\left(\left([Y_i,Y_j]\right)^T-\left(\nabla_{Y_j}Y_i\right)^\bot\right)\right)^\bot=0\\
\end{array}
\end{equation}
for all $1\leqslant j \leqslant n$.

\section{Lie Groups with Biinvariant Metric}\label{ch3}

In this section we consider a Lie group $N$ with some biinvariant metric. The conditions from Theorem~\ref{th1} in this
particular case are relatively simple:

\begin{proposition}\label{pr1}
The Gauss map of a smooth submanifold $M$ in the Lie group $N$ with biinvariant metric $\langle \cdot,\cdot \rangle$ is
harmonic at a point $p \in M$ if and only if, in the above notation,
\begin{equation}\label{eq2}
\begin{array}{c}
\langle [nH,Y_j],Y_\alpha\rangle +\sum\limits_{1 \leqslant i \leqslant n,n+1\leqslant \gamma\leqslant
n+q}b^\gamma_{ij}\langle[Y_i,Y_\gamma],Y_\alpha\rangle\\
+\frac{1}{2}\sum\limits_{1\leqslant i \leqslant n}\langle
[Y_i,Y_j]^\bot,[Y_i,Y_\alpha]^\bot \rangle=0\\
\end{array}
\end{equation}
for $1\leqslant j\leqslant n$, $n+1\leqslant \alpha\leqslant n+q$.
\end{proposition}
\begin{proof}
Recall that the left invariant metric $\langle \cdot,\cdot\rangle$ is biinvariant if and only if $\langle
[X,Y],Z\rangle=\langle X,[Y,Z]\rangle$ for all left invariant $X$, $Y$, and $Z$. Also, $\nabla_X Y=\frac{1}{2}[X,Y]$.
In particular, $\nabla_X Y=-\nabla_Y X$ and $\nabla_X X=0$. This, together with the symmetry of the second fundamental
form, implies $\sum\limits_{1\leqslant i,k \leqslant n}b_{ik}^\alpha\langle\nabla_{Y_i}Y_k, Y_j \rangle=0$. The
curvature tensor is defined by the equation $\riemann(X,Y)Z=-\frac{1}{4}[[X,Y],Z]$. Thus,
$$
\begin{array}{c}
-\sum\limits_{1\leqslant i\leqslant n}\langle\left(\nabla_{Y_i}Y_j\right)^T,\left(\nabla_{Y_i}Y_\alpha\right)^T
\rangle+\sum\limits_{1\leqslant i\leqslant n}\langle\left(\nabla_{Y_i}Y_j\right)^\bot,
\left(\nabla_{Y_i}Y_\alpha\right)^\bot \rangle\\
=-\frac{1}{4}\sum\limits_{1\leqslant i\leqslant n}\langle[Y_i,Y_j]^T,[Y_i,Y_\alpha]^T
\rangle+\frac{1}{4}\sum\limits_{1\leqslant i\leqslant n}\langle [Y_i,Y_j]^\bot, [Y_i,Y_\alpha]^\bot \rangle\\
=-\frac{1}{4}\sum\limits_{1\leqslant i\leqslant n}\langle[Y_i,Y_j],[Y_i,Y_\alpha]
\rangle+\frac{1}{2}\sum\limits_{1\leqslant i\leqslant n}\langle [Y_i,Y_j]^\bot, [Y_i,Y_\alpha]^\bot \rangle\\
=\sum\limits_{1\leqslant i\leqslant n}\langle\riemann(Y_i,Y_j),Y_i),Y_\alpha \rangle+\frac{1}{2}\sum\limits_{1\leqslant
i\leqslant n}\langle [Y_i,Y_j]^\bot, [Y_i,Y_\alpha]^\bot \rangle.\\
\end{array}
$$
Substitute this in \eqref{eq1} and obtain \eqref{eq2}.
\end{proof}

If $q=1$ (i.e., $M$ is a hypersurface), then $\langle [Y_i,Y_{n+1}],Y_{n+1}\rangle=\langle
Y_i,[Y_{n+1},Y_{n+1}]\rangle$ vanishes for all $1\leqslant i\leqslant n$, i.e., $[Y_i,Y_{n+1}]^\bot=0$. It follows that
\eqref{eq2} gives the conditions $Y_j(nH)=0$, where $H$ is the mean curvature function. This implies the result from
\cite{ES} cited in the introduction.

Denote by $\mathcal{N}$ the Lie algebra of $N$. It is well-known (see, for example, \cite{Mi}, Lemma~7.5), that
$\mathcal{N}$ is compact, i.e., $\mathcal{N}=\mathcal{Z}\oplus\mathcal{N}'$, where the direct sum is orthogonal,
$\mathcal{Z}$ is abelian, and $\mathcal{N}'=[\mathcal{N},\mathcal{N}]$ is semisimple with the negative definite Killing
form.

Let $M$ be totally geodesic submanifold of $N$, $\Psi \colon M \rightarrow N$ the corresponding immersion, and $p$ an
arbitrary point of $M$. Consider the immersion $\Psi'=L_{\Psi(p)^{-1}}\circ \Psi \colon M \rightarrow N$. The image
$\Psi'(p)$ coincides with the identity element $e$ of the group. The Gauss map of this immersion maps each point $r\in
M$ to the subspace $\Phi'(r)=dL_{\Psi'(p)^{-1}}\circ d\Psi'(T_r M)=dL_{\Psi(p)^{-1}}\circ d\Psi(T_r M)=\Phi(r)$, i.e.,
the Gauss maps of two immersions are the same. Left translations are isometries of $N$, hence $\Psi'$ is also totally
geodesic. Thus we can assume without loss of generality that $\Psi(p)=e$. Then the tangent space $T_e M$ is a Lie
triple system in $\mathcal{N}$ (see, for example, \cite{KN}, Theorem~4.3 of Chapter XI). The subspace
$\overline{\mathcal{N}}=T_e M+[T_e M,T_e M]$ is a compact Lie subalgebra, therefore it has an orthogonal direct
decomposition $\overline{\mathcal{N}}=\overline{\mathcal{Z}}\oplus\overline{\mathcal{N}}'$ with abelian
$\overline{\mathcal{Z}}$ and semisimple $\overline{\mathcal{N}}'=[\overline{\mathcal{N}},\overline{\mathcal{N}}]$. Take
the decomposition $Y_a=X_a+Z_a$ for $1 \leqslant a \leqslant n+q_1$, where $X_a\in \overline{\mathcal{N}}'$,
$Z_a\in\overline{\mathcal{Z}}$, $\dim \overline{\mathcal{N}} = n+q_1$. Then for $1 \leqslant a,b \leqslant n+q_1$ the Lie bracket $[Y_a,Y_b]=[X_a,X_b]$. Denote
by $\mathcal{W}$ the subspace spanned by $X_1,\dots,X_n$ (i.e., the orthogonal projection of $T_e M$ to
$\overline{\mathcal{N}}'$). It is a Lie triple system in $\overline{\mathcal{N}}'$, and
$\overline{\mathcal{N}}'=\mathcal{W}+[\mathcal{W},\mathcal{W}]$. The intersection
$\overline{\mathcal{W}}=\mathcal{W}\cap[\mathcal{W},\mathcal{W}]$ is an ideal (from this point on by ideals we mean
ideals in $\overline{\mathcal{N}}$). The Lie algebra $\overline{\mathcal{N}}'$ is semisimple, consequently the
orthogonal complement $\mathcal{V}$ to $\overline{\mathcal{W}}$ is an ideal and equals an orthogonal direct sum
$\bigoplus\limits_{1\leqslant l \leqslant m}\mathcal{S}_l$ of simple ideals $\mathcal{S}_l$.

\begin{theorem}\label{pr2}
Let $M$ be a smooth immersed totally geodesic submanifold in a Lie group $N$ with biinvariant metric. Then
\begin{enumerate}

\item\label{pr2-it1} if the restriction of the metric to $\mathcal{V}$ is a negative multiple of the Killing form (in
particular, if $\mathcal{V}$ is simple), then the Gauss map of $M$ in this metric is harmonic;

\item\label{pr2-it2} if $\mathcal{W}\cap\mathcal{V}=\bigoplus\limits_{1\leqslant l \leqslant m}\mathcal{W}_l$, where
$\mathcal{W}_l\subset\mathcal{S}_l$ is a proper Lie triple system in $\mathcal{S}_l$, i.e., $\mathcal{W}_l \neq 0$ and
$\mathcal{W}_l \neq \mathcal{S}_l$ for each $1\leqslant l \leqslant m$ (in particular, if $\mathcal{V}=0$), then the
Gauss map of $M$ is harmonic in any biinvariant metric on $N$;

\item\label{pr2-it3} if the condition of \ref{pr2-it2} is not satisfied, then there is a biinvariant metric on $N$ such
that the Gauss map of $M$ is not harmonic.

\end{enumerate}
\end{theorem}
\begin{proof}
The conditions \eqref{eq2} for $1\leqslant j\leqslant n$, $n+1\leqslant \alpha\leqslant n+q$ take the form
\begin{equation}\label{pr2-eq1-1}
\begin{array}{c}
\sum\limits_{1\leqslant i \leqslant n}\langle [Y_i,Y_j]^\bot,[Y_i,Y_\alpha]^\bot \rangle=0.\\
\end{array}
\end{equation}
Also note that
$$
\begin{array}{c}
\sum\limits_{1\leqslant i \leqslant n}\langle [Y_i,Y_j]^T,[Y_i,Y_\alpha]^T \rangle+\sum\limits_{1\leqslant i \leqslant
n}\langle [Y_i,Y_j]^\bot,[Y_i,Y_\alpha]^\bot \rangle\\
=\sum\limits_{1\leqslant i \leqslant n}\langle [Y_i,Y_j],[Y_i,Y_\alpha]\rangle=\sum\limits_{1\leqslant i \leqslant
n}\langle [[Y_i,Y_j],Y_i],Y_\alpha\rangle=0\\
\end{array}
$$
since the tangent space $T_e M$ is a Lie triple system. Hence the conditions \eqref{pr2-eq1-1} are equivalent to
\begin{equation}\label{pr2-eq1-2}
\begin{array}{c}
\sum\limits_{1\leqslant i \leqslant n}\langle [Y_i,Y_j]^T,[Y_i,Y_\alpha]^T \rangle=0.\\
\end{array}
\end{equation}
Note that \eqref{pr2-eq1-1} and \eqref{pr2-eq1-2} can also be obtained directly from \eqref{eq1-2} using the fact that
$T_e M$ is a Lie triple system and the expression for the invariant Riemannian connection.

The ideal $\overline{\mathcal{W}}$ is semisimple, hence
$\overline{\mathcal{W}}=[\overline{\mathcal{W}},\overline{\mathcal{W}}]$ and
$$
\begin{array}{c}
\overline{\mathcal{W}}=[\overline{\mathcal{W}}, \overline{\mathcal{W}}]=[\overline{\mathcal{W}},
[\overline{\mathcal{W}},\overline{\mathcal{W}}]] \subset[\mathcal{W}, [\mathcal{W},\mathcal{W}]]=[T_e M, [T_e M,T_e
M]]\subset T_e M.\\
\end{array}
$$
This implies that we can choose a frame of $T_e M$ such that $Y_i=X_i\in \overline{\mathcal{W}}$ for $1\leqslant i
\leqslant n_1$, where $0 \leqslant n_1 \leqslant n$, and $Y_i=X_i+Z_i$ with $X_i
\in\widetilde{\mathcal{W}}=\mathcal{W}\cap\mathcal{V}$ and $Z\in\overline{\mathcal{Z}}$ for $n_1+1\leqslant i \leqslant
n$. For $1\leqslant j \leqslant n_1$ the equations in \eqref{pr2-eq1-1} become
$$
\begin{array}{c}
\sum\limits_{1\leqslant i \leqslant n}\langle [Y_i,Y_j]^\bot,[Y_i,Y_\alpha]^\bot \rangle=\sum\limits_{1\leqslant i
\leqslant n_1}\langle [Y_i,Y_j]^\bot,[Y_i,Y_\alpha]^\bot \rangle=0\\
\end{array}
$$
because $[Y_i,Y_j]\in T_e M$ for $1\leqslant i \leqslant n_1$. This yields that for showing harmonicity or
non-harmonicity of the Gauss map at the point it suffices to check \eqref{pr2-eq1-1} or \eqref{pr2-eq1-2} for
$n_1+1\leqslant j \leqslant n$.

The subspace $\widetilde{\mathcal{W}}$ is a Lie triple system in a semisimple Lie algebra $\mathcal{V}$, and
$\mathcal{V}=\widetilde{\mathcal{W}}+[\widetilde{\mathcal{W}},\widetilde{\mathcal{W}}]$. Moreover,
$\widetilde{\mathcal{W}}\cap[\widetilde{\mathcal{W}},\widetilde{\mathcal{W}}]=0$ because $\mathcal{V}$ is a direct
complement to $\mathcal{W}\cap[\mathcal{W},\mathcal{W}]$. For each $1\leqslant l\leqslant m$ the restriction of the
inner product to $\mathcal{S}_l$ is equal to the Killing form multiplied by a negative constant: $\langle X,Y
\rangle=\lambda_l \trace(\ad X \circ \ad Y)$ for $X,Y \in \mathcal{S}_l$, $\lambda_l<0$ (See \cite{Mi}, Lemma~7.6).
Here by $\ad X$ we mean the restriction of the adjoint representation operator to the corresponding simple ideal.
Denote by $P_l$ the orthogonal projection to $\mathcal{S}_l$, then $\langle X,Y \rangle=\sum\limits_{1\leqslant l
\leqslant m}\lambda_l \trace(\ad P_l(X) \circ \ad P_l(Y))$ for $X,Y \in \mathcal{V}$.

For each $1\leqslant l\leqslant m$ the operator $P_l$ is a Lie algebra homomorphism, therefore
$P_l([\widetilde{\mathcal{W}},\widetilde{\mathcal{W}}])=[P_l(\widetilde{\mathcal{W}}),P_l(\widetilde{\mathcal{W}})]$
and
$$
\begin{array}{c}
\mathcal{S}_l=P_l(\mathcal{V})=P_l(\widetilde{\mathcal{W}}+[\widetilde{\mathcal{W}},\widetilde{\mathcal{W}}])
=P_l(\widetilde{\mathcal{W}})+[P_l(\widetilde{\mathcal{W}}),P_l(\widetilde{\mathcal{W}})].\\
\end{array}
$$
The intersection $P_l(\widetilde{\mathcal{W}})\cap[P_l(\widetilde{\mathcal{W}}),P_l(\widetilde{\mathcal{W}})]$ is an
ideal in simple $\mathcal{S}_l$. Hence either
$P_l(\widetilde{\mathcal{W}})\cap[P_l(\widetilde{\mathcal{W}}),P_l(\widetilde{\mathcal{W}})]=0$ or
$\mathcal{S}_l=P_l(\widetilde{\mathcal{W}})=[P_l(\widetilde{\mathcal{W}}),P_l(\widetilde{\mathcal{W}})]$. In the first
case the operators $\ad X$ for $X \in P_l(\widetilde{\mathcal{W}})$ map $P_l(\widetilde{\mathcal{W}})$ to
$[P_l(\widetilde{\mathcal{W}}),P_l(\widetilde{\mathcal{W}})]$ and
$[P_l(\widetilde{\mathcal{W}}),P_l(\widetilde{\mathcal{W}})]$ to $P_l(\widetilde{\mathcal{W}})$. The operators $\ad Y$
for $Y \in [P_l(\widetilde{\mathcal{W}}),P_l(\widetilde{\mathcal{W}})]$ map the subspaces
$P_l(\widetilde{\mathcal{W}})$ and $[P_l(\widetilde{\mathcal{W}}),P_l(\widetilde{\mathcal{W}})]$ to themselves. It
follows that $\langle P_l(\widetilde{\mathcal{W}}),
[P_l(\widetilde{\mathcal{W}}),P_l(\widetilde{\mathcal{W}})]\rangle=0$.

If the restriction of the metric to $\mathcal{V}$ is a negative multiple of the Killing form (the case of
\ref{pr2-it1}), then the same argument shows that $\langle \widetilde{\mathcal{W}},
[\widetilde{\mathcal{W}},\widetilde{\mathcal{W}}]\rangle=0$.

Consider the case $P_l(\widetilde{\mathcal{W}})\cap[P_l(\widetilde{\mathcal{W}}),P_l(\widetilde{\mathcal{W}})]=0$ for
all $1\leqslant l\leqslant m$. We proved that $\langle
P_l(\widetilde{\mathcal{W}}),[P_l(\widetilde{\mathcal{W}}),P_l(\widetilde{\mathcal{W}})]\rangle=0$ for all $l$, thus
$\langle \widetilde{\mathcal{W}},[\widetilde{\mathcal{W}},\widetilde{\mathcal{W}}]\rangle=0$. For each $1\leqslant
l\leqslant m$ denote $P_l(\widetilde{\mathcal{W}})$ by $\mathcal{W}_l$. Then
$\langle\mathcal{W}_l,[\widetilde{\mathcal{W}},\widetilde{\mathcal{W}}]\rangle
=\langle\mathcal{W}_l,[\mathcal{W}_l,\mathcal{W}_l]\rangle=0$, hence $\mathcal{W}_l$ is contained in the orthogonal
complement of $[\widetilde{\mathcal{W}},\widetilde{\mathcal{W}}]$, i.e., in $\widetilde{\mathcal{W}}$; and
$\widetilde{\mathcal{W}}=\bigoplus\limits_{1\leqslant l \leqslant m}\mathcal{W}_l$. Subspaces $\mathcal{W}_l$ are Lie
triple systems, $\mathcal{W}_l\neq 0$ because in the opposite case $[\mathcal{W}_l,\mathcal{W}_l]=0$ and
$\mathcal{S}_l=0$. It contradicts the fact that $\mathcal{S}_l$ is simple. If $\mathcal{W}_l=\mathcal{S}_l$, then
$\mathcal{W}_l=[\mathcal{W}_l,\mathcal{W}_l]$ because $\mathcal{S}_l$ is simple, a contradiction. It follows that
$\mathcal{W}_l\neq\mathcal{S}_l$. This is the case of \ref{pr2-it2}. It is easy to see also that the condition in
\ref{pr2-it2} implies $P_l(\widetilde{\mathcal{W}})\cap[P_l(\widetilde{\mathcal{W}}),P_l(\widetilde{\mathcal{W}})]=0$
for all $1\leqslant l\leqslant m$. In fact, if $\widetilde{\mathcal{W}}=\bigoplus\limits_{1\leqslant l \leqslant
m}\mathcal{W}_l$ with $\mathcal{W}_l \subset \mathcal{S}_l$, then $P_l(\widetilde{\mathcal{W}})=\mathcal{W}_l$,
$[P_l(\widetilde{\mathcal{W}}),P_l(\widetilde{\mathcal{W}})]=[\mathcal{W}_l,\mathcal{W}_l]$, therefore the case
$\mathcal{S}_l=P_l(\widetilde{\mathcal{W}})=[P_l(\widetilde{\mathcal{W}}),P_l(\widetilde{\mathcal{W}})]$ is excluded by
the condition $\mathcal{W}_l \neq \mathcal{S}_l$.

Assume that $\langle \widetilde{\mathcal{W}}, [\widetilde{\mathcal{W}},\widetilde{\mathcal{W}}]\rangle=0$. Take any
$n_1+1\leqslant j \leqslant n$. For $1\leqslant i\leqslant n_1$ $[Y_i,Y_j]=0$ and for $n_1+1\leqslant i \leqslant n$
$$
\begin{array}{c}
[Y_i,Y_j]^T=\sum\limits_{1\leqslant k \leqslant n}\langle[Y_i,Y_j],Y_k\rangle Y_k=\sum\limits_{n_1+1\leqslant k
\leqslant n}\langle[X_i,X_j],X_k\rangle Y_k=0\\
\end{array}
$$
because $X_i \in\widetilde{\mathcal{W}}$ for $n_1+1\leqslant i \leqslant n$. This yields that \eqref{pr2-eq1-2} is
satisfied. We proved \ref{pr2-it1} and \ref{pr2-it2}.

Finally, in the case \ref{pr2-it3} there is $1\leqslant l_0\leqslant m$ such that
$\mathcal{S}_{l_0}=P_{l_0}(\widetilde{\mathcal{W}})=[P_{l_0}(\widetilde{\mathcal{W}}),P_{l_0}(\widetilde{\mathcal{W}})]$.
Consider the new metric $\langle\cdot,\cdot\rangle'$ such that it is equal to $\langle\cdot,\cdot\rangle$ on the
orthogonal complement to $\mathcal{V}$ and
$$
\begin{array}{c}
\langle X,Y \rangle'=\sum\limits_{1\leqslant l \leqslant m}\left(-\trace(\ad P_l(X) \circ \ad
P_l(Y))\right)-\lambda^2\trace(\ad P_{l_0}(X) \circ \ad P_{l_0}(Y))\\
\end{array}
$$
for $X,Y \in \mathcal{V}$, where $\lambda \neq 0$. It is a biinvariant metric. Denote $-\trace(\ad P_{l_0}(X) \circ \ad
P_{l_0}(Y))$ by $\langle X,Y\rangle''$.

The ideal $\mathcal{S}_{l_0}$ is not contained in $\widetilde{\mathcal{W}}$ because in the opposite case it is
contained also in $\widetilde{\mathcal{W}}\cap[\widetilde{\mathcal{W}},\widetilde{\mathcal{W}}]=0$, a contradiction. It
follows that there is a vector $Y$ orthogonal to $\widetilde{\mathcal{W}}$ such that $P_{l_0}(Y)\neq 0$. Then
$\mathcal{S}_{l_0}=P_{l_0}(\widetilde{\mathcal{W}})$ implies that there is a vector $X \in \widetilde{\mathcal{W}}$
such that $P_{l_0}(X)=P_{l_0}(Y)$. Note that $Y$ is orthogonal to $T_e M$. We can consider that the norm of $Y$ equals
$1$. Choose an orthonormal frames of $T_e M$ and $N_e M$ such that $Y_{j_0}=X+Z_{j_0}$ for some $n_1+1\leqslant j_0
\leqslant n$ and $Y_{\alpha_0}=Y$ for some $n+1\leqslant \alpha_0 \leqslant n+q_1$. Then the discussion above implies
that for any $n_1+1\leqslant i \leqslant n$ in the new metric
$$
\begin{array}{c}
[Y_i,Y_{j_0}]^T=\sum\limits_{1\leqslant k \leqslant n}\langle[Y_i,Y_{j_0}],Y_k\rangle' Y_k=\sum\limits_{n_1+1\leqslant
k \leqslant n}\langle[X_i,X],X_k\rangle' Y_k\\
=\lambda^2\sum\limits_{n_1+1\leqslant k \leqslant n}\langle[P_{l_0}(X_i),P_{l_0}(X)],P_{l_0}(X_k)\rangle'' Y_k\\
=-\lambda^2\sum\limits_{n_1+1\leqslant k \leqslant n}\langle[P_{l_0}(X_i),P_{l_0}(X_k)],P_{l_0}(X)\rangle'' Y_k.\\
\end{array}
$$
There is some $n_1+1\leqslant i \leqslant n$ such that this expression does not vanish because
$\mathcal{S}_{l_0}=[P_{l_0}(\widetilde{\mathcal{W}}),P_{l_0}(\widetilde{\mathcal{W}})]$. Similarly,
$$
\begin{array}{c}
[Y_i,Y_{\alpha_0}]^T
=-\lambda^2\sum\limits_{n_1+1\leqslant k \leqslant n}\langle[P_{l_0}(X_i),P_{l_0}(X_k)],P_{l_0}(Y)\rangle'' Y_k.\\
\end{array}
$$
The expression in \eqref{pr2-eq1-2} for $j=j_0$ and $\alpha=\alpha_0$ thus becomes
$$
\begin{array}{c}
\lambda^4\sum\limits_{n_1+1\leqslant i,k \leqslant n}\left(\langle [P_{l_0}(X_i),P_{l_0}(X_k)],
P_{l_0}(X)\rangle''\right)^2 \neq 0.\\
\end{array}
$$
Therefore the Gauss map is not harmonic.
\end{proof}

A Lie triple system $\mathcal{U}$ is {\itshape reducible} if $\mathcal{U}=\mathcal{U}_1\oplus\mathcal{U}_2$, where
$\mathcal{U}_1$ and $\mathcal{U}_2$ are nonzero Lie triple systems such that $[\mathcal{U}_1,\mathcal{U}_2]=0$, and is
{\itshape irreducible} otherwise (see, for example, Appendix~1 of \cite{E3}). Theorem~\ref{pr2} then implies that if
$\widetilde{\mathcal{W}}$ is irreducible and $\mathcal{V}$ is not simple, then there is a biinvariant metric on $N$
such that the Gauss map of $M$ is not harmonic.

Consider an example. Let $\mathcal{N}$ be $\mathfrak{so}(3)\oplus\mathfrak{so}(3)$ with the orthogonal basis consisting
of the vectors $e_1,e_2,e_3,f_1,f_2,f_3$ with the nonzero brackets
$$
\begin{array}{c}
\vphantom{\sum\limits_{1\leqslant i \leqslant n}}
[e_1,e_2]=-[e_2,e_1]=e_3,[e_2,e_3]=-[e_3,e_2]=e_1,[e_3,e_1]=-[e_1,e_3]=e_2,\\
\vphantom{\sum\limits_{1\leqslant i \leqslant n}}
[f_1,f_2]=-[f_2,f_1]=f_3,[f_2,f_3]=-[f_3,f_2]=f_1,[f_3,f_1]=-[f_1,f_3]=f_2.\\
\end{array}
$$
Let $\mathcal{W}$ be the subspace spanned by $e_1+f_1$, $e_2-f_2$, and $e_3+f_3$. Let $M$ be $\exp(\mathcal{W})$, hence
$T_e M=\mathcal{W}$. The bracket $[\mathcal{W},\mathcal{W}]$ is spanned by $e_1-f_1$, $e_2+f_2$, and $e_3-f_3$. It is
easy to see that $\mathcal{W}$ is a Lie triple system. In our notation,
$\mathcal{N}=\overline{\mathcal{N}}=\overline{\mathcal{N}}'=\mathcal{V}=\mathcal{W}+[\mathcal{W},\mathcal{W}]$. The
intersection $\overline{\mathcal{W}}=\mathcal{W}\cap[\mathcal{W},\mathcal{W}]$ vanishes, therefore
$\widetilde{\mathcal{W}}=\mathcal{W}$. Choose a metric such that $\langle e_i,e_j\rangle=\delta_{ij}$ and $\langle
f_i,f_j\rangle=\delta_{ij}a^2$, where $0 < a \neq 1$, then $\mathcal{W}$ and $[\mathcal{W},\mathcal{W}]$ are not
orthogonal. The orthonormal frames of the tangent and the normal spaces of $M$ can be chosen in the following way:
$$
\begin{array}{c}
\vphantom{\sum\limits_{1\leqslant i \leqslant n}}
Y_1=\frac{1}{\sqrt{1+a^2}}(e_1+f_1),Y_2=\frac{1}{\sqrt{1+a^2}}(e_2-f_2),Y_3=\frac{1}{\sqrt{1+a^2}}(e_3+f_3),\\
\vphantom{\sum\limits_{1\leqslant i \leqslant n}}
Y_4=\frac{\sqrt{1+a^2}}{a}\left(e_1-\frac{1}{a^2}f_1\right),
Y_5=\frac{\sqrt{1+a^2}}{a}\left(e_2+\frac{1}{a^2}f_2\right),
Y_6=\frac{\sqrt{1+a^2}}{a}\left(e_3-\frac{1}{a^2}f_3\right).\\
\end{array}
$$
Compute \eqref{pr2-eq1-2}, e.g., for $j=1$ and $\alpha=4$:
$$
\begin{array}{c}
\sum\limits_{1\leqslant i \leqslant 3}\langle [Y_i,Y_1]^T,[Y_i,Y_4]^T \rangle=\frac{1}{a(1+a^2)}\langle
\left(-e_3+f_3\right)^T,\left(-e_3-\frac{1}{a^2}f_3\right)^T \rangle\\
\vphantom{\sum\limits_{1\leqslant i \leqslant n}} +\frac{1}{a(1+a^2)}\langle
\left(e_2+f_2\right)^T,\left(e_2-\frac{1}{a^2}f_2\right)^T
\rangle=\frac{-2(-1+a^2)}{a(1+a^2)^2}+\frac{2(1-a^2)}{a(1+a^2)^2}\neq 0.
\end{array}
$$
It follows that the Gauss map is not harmonic.

\section{$2$-Step Nilpotent Groups and Geodesics}\label{ch4}

Recall that a Lie group $N$ is {\itshape $2$-step nilpotent} if and only if its Lie algebra $\mathcal{N}$ is $2$-step
nilpotent, i.e., $[\mathcal{N},\mathcal{N}]\neq 0$, $[[\mathcal{N},\mathcal{N}],\mathcal{N}]=0$. In other words, $0\neq
[\mathcal{N},\mathcal{N}]\subset \mathcal{Z}$, where $\mathcal{Z}$ is the center of $\mathcal{N}$. Consider a $2$-step
nilpotent Lie group $N$ with left invariant metric induced by an inner product $\langle\cdot,\cdot\rangle$ on
$\mathcal{N}$ as above. Denote by $\mathcal{V}$ the orthogonal complement to $\mathcal{Z}$ in $\mathcal{N}$. For each
$Z \in \mathcal{Z}$ define a linear operator $J(Z)\colon \mathcal{V} \rightarrow \mathcal{V}$ by $\langle
J(Z)X,Y\rangle=\langle[X,Y],Z\rangle$ for all $X$, $Y$ from $\mathcal{V}$. All $J(Z)$ are skew-symmetric. The group $N$
and the Lie algebra $\mathcal{N}$ are called {\itshape nonsingular} if for each $Z \neq 0$ the operator $J(Z)$ is
nondegenerate.

The left invariant Riemannian connection is defined by (see \cite{E1})
\begin{equation}\label{2st-1}
\begin{array}{lcll}
\nabla_{X}Y&=&\frac{1}{2}[X,Y], & X,Y \in \mathcal{V};\\
\nabla_{X}Z=\nabla_{Z}X&=&-\frac{1}{2}J(Z)X, & X \in \mathcal{V}, \, Z \in \mathcal{Z};\\
\nabla_{Z}Z^*&=&0, & Z,Z^* \in \mathcal{Z}.\\
\end{array}
\end{equation}

Let us investigate whether the Gauss map of a totally geodesic submanifold $M$ in $N$ is harmonic. It was proved in
\cite{E2} (Theorem~(4.2)) that if $N$ is simply connected and nonsingular, then a totally geodesic submanifold of dimension $n geqslant 2$ either have the Gauss map of maximal rank at any point or it is a left translation of some open subset in a totally geodesic subgroup. The latter case takes place for many classes of submanifolds, for example, for all totally geodesic $M$ such that $\dim M>\dim \mathcal{Z}$ in $2$-step nilpotent groups $N$ with $\dim N\geqslant 3$ (see \cite{E2}, Corollary~(5.6)). The structure of the corresponding subgroups (or their Lie algebras) is also described in \cite{E2} (and allows to prove, for example, that there are no totally geodesic hypersurfaces in nonsingular $2$-step nilpotent Lie groups, see \cite{E2}, Corollary~(5.8)). Anyway, in this case the Gauss map is constant, thus harmonic. Therefore it suffices to consider the case of the Gauss map with maximal rank.

For $n=\dim M=1$, i.e., for geodesics, the answer is given by the next statement:

\begin{proposition}\label{pr3}
A smooth geodesic in $2$-step nilpotent group has the harmonic Gauss map if and only if it is a left translation of
some one-parameter subgroup.
\end{proposition}
\begin{proof}
The "if" part is clear, let us prove the "only if" part. Taking if necessary a left translation we can think that our
geodesic contains the identity $e$ of $N$ (similarly to the discussion in the previous section). Decompose its tangent
vector at $e$ as $X+Z$, where $X\in \mathcal{V}$ and $Z \in \mathcal{Z}$. The condition \eqref{eq1-2} with $n=1$,
$j=1$, and $Y_1=X+Z$ becomes
$$
\begin{array}{c}
0=\left([X+Z,-J(Z)X]+2\nabla_{X+Z}\left(J(Z)X\right)^\bot\right)^\bot\\
=\left([J(Z)X,X]+[X,J(Z)X]-J(Z)^2 X\right)^\bot=-\left(J(Z)^2 X\right)^\bot.\\
\end{array}
$$
Here we used the definition of $2$-step nilpotent Lie algebra, the equations \eqref{2st-1}, and the fact that $\langle
J(Z)X,X+Z\rangle=\langle J(Z)X,X\rangle=0$ because $J(Z)$ is skew-symmetric, therefore $\left(J(Z)X\right)^\bot=J(Z)X$.
The conditions mean $J(Z)^2 X=\lambda(X+Z)$, where $\lambda \in \mathbb{R}$. Thus $\lambda Z=0$, hence $Z=0$ or
$\lambda=0$, in any case $J(Z)^2 X=0$. This yields $0=\langle J(Z)^2 X,X\rangle=-|J(Z)X|^2$, therefore $J(Z)X$
vanishes. Then Proposition~(3.5) of \cite{E1} implies that the geodesic is defined by the formula $\exp(t(X+Z))$. This
gives us the desired result.
\end{proof}

Actually, the proof implies that the geodesic is a left translation of one-parameter subgroup if the Gauss map is
harmonic only at some point. Anyway, it follows that for $n=1$ the Gauss maps of maximal rank are not harmonic. It is
interesting to check whether the similar statement is true for other values of $n$.

\end{document}